\documentclass[11pt]{amsart}
\usepackage{amsmath,amssymb,,latexsym,esint,cite}
\usepackage{verbatim}
\usepackage[left=2.65cm,right=2.65cm,top=3cm,bottom=3cm]{geometry}

\usepackage{color,enumitem,graphicx}
\usepackage[colorlinks=true,urlcolor=blue, citecolor=red,linkcolor=blue,linktocpage,pdfpagelabels,bookmarksnumbered,bookmarksopen]{hyperref}

\usepackage[hyperpageref]{backref}
\usepackage[english]{babel}

\newcommand{\abs}[1]{\left|#1\right|}
\newcommand{\bdry}[1]{\partial #1}
\newcommand{\bgset}[1]{\big\{#1\big\}}

\newcommand{\F}{{\mathcal F}}
\newcommand{\closure}[1]{\overline{#1}}

\newcommand{\incl}{\subset}
\newcommand{\M}{{\mathcal M}}
\newcommand{\N}{\mathbb N}
\newcommand{\norm}[2][]{\left\|#2\right\|_{#1}}

\newcommand{\PS}[1]{$(\text{PS})_{#1}$}
\newcommand{\pnorm}[2][]{\if #1'' \left|#2\right|_p \else \left|#2\right|_{#1} \fi}
\newcommand{\R}{\mathbb R}
\newcommand{\RP}{\R \text{P}}

\newcommand{\seq}[1]{\left(#1\right)}
\newcommand{\set}[1]{\left\{#1\right\}}

\newcommand{\Z}{\mathbb Z}
\newcommand{\isom}{\approx}
\DeclareMathOperator{\divg}{div}
\renewcommand{\H}{{\mathcal H}}
\newcommand{\wto}{\rightharpoonup}
\renewcommand{\o}{\text{o}}

\newtheorem{lemma}{Lemma}[section]
\newtheorem{proposition}[lemma]{Proposition}
\newtheorem{theorem}[lemma]{Theorem}

\theoremstyle{definition}

\theoremstyle{remark}

\theoremstyle{definition}
\newtheorem{definition}[lemma]{Definition}

\newenvironment{enumarab}{\begin{enumerate}
		
		}{\end{enumerate}}

\newenvironment{enumroman}{\begin{enumerate}
		
		}{\end{enumerate}}

\numberwithin{equation}{section}

\title[Double-phase problems via Morse theory]{Existence results for double-phase \\ problems via Morse theory}

\author[K.\ Perera]{Kanishka Perera}
\author[M.\ Squassina]{Marco Squassina}

\address[K. Perera]{Department of Mathematical Sciences
	\newline\indent
	Florida Institute of Technology
	\newline\indent
	150 W University Blvd, Melbourne, FL 32901, USA}
\email{kperera@fit.edu}

\address[M.\ Squassina]{Dipartimento di Informatica
	\newline\indent
	Universit\`a degli Studi di Verona
	\newline\indent
	Verona, Italy}
\email{marco.squassina@univr.it}


\begin{document}

\begin{abstract}
We obtain nontrivial solutions for a class of double-phase problems using Morse theory. In the absence of a direct sum decomposition, we use a cohomological local splitting to get an estimate of the critical groups at zero.
\end{abstract}

\maketitle

%

\section{Introduction}
The study of energy functionals of the form
 \begin{equation} 
 	\label{doublef-energy}
 	u\mapsto \int_{\Omega} {\mathcal H}(x,|\nabla u(x)|)dx,\qquad
 	{\mathcal H}(x,t)=t^p+a(x)t^q,\quad q>p>1,\,\,\,  a(\cdot)\geq 0,
 \end{equation}
 where the integrand ${\mathcal H}$ 
 switches between two different elliptic behaviors
 has been intensively studied since the late eighties. This class of
 energies was introduced by Zhikov 
 to provide models of {\em strongly anisotropic} materials, see e.g.\ \cite{zhikov86,zhikov95,zhikov97} or the monograph \cite{zhikovBook}.
 Also,  the integrals of \eqref{doublef-energy} settle in the framework of 
 the so-called functionals with non-standard growth conditions,
 according to a terminology introduced by Marcellini \cite{Marcellini89,Marcellini91,cup5}.\
 In \cite{zhikovBook}, energies of the form \eqref{doublef-energy} are
 used in the context of homogenization and elasticity  and $a(\cdot)$ drives the geometry of a composite of two
 different materials with hardening powers $p$ and $q$.
 
 Significant progresses were recently achieved by the school of Mingione in the framework of regularity theory for  minimizers of \eqref{doublef-energy}, see 
 e.g.\ \cite{BarColMin1,BarColMin2,BarKuuMin,ColMing1,ColMing2}.
More recently, in \cite{CoSq}, a complete study on the existence and properties
of a sequence of variational eigenvalues related to ${\mathcal H}$, including
a Weyl type estimate for their growth has been performed.
\vskip2pt
\noindent
The purpose of this paper is to investigate the existence of solutions to the double phase problem
\begin{equation} \label{1}
\left\{\begin{aligned}
- \divg \left(|\nabla u|^{p-2}\, \nabla u + a(x)\, |\nabla u|^{q-2}\, \nabla u\right) & = f(x,u) && \text{in } \Omega\\
u & = 0 && \text{on } \bdry{\Omega},
\end{aligned}\right.
\end{equation}
where $\Omega\subset\R^N$ is a bounded domain with Lipschitz boundary, $N \ge 2$, $1 < p < q < N$,
\begin{equation} 
\label{p-q-cc}
q/p < 1 + 1/N,\quad \text{$a : \closure{\Omega} \to [0,\infty)$ is Lipschitz continuous},
\end{equation}
and $f$ is a Carath\'{e}odory function on $\Omega \times \R$ satisfying the growth condition
\begin{equation} 
\label{2}
|f(x,t)| \le C \left(|t|^{r-1} + 1\right) \quad \text{for a.a. } x \in \Omega \text{ and all } t \in \R,
\end{equation}
for some $1<r<p^\ast$ and $C > 0$, being $p^\ast = Np/(N - p)$
the critical Sobolev exponent of $W^{1,p}_0(\Omega)$. Assuming that $f(x,0) \equiv 0$, problem \eqref{1} has the trivial solution $u(x) =0$ and we study the critical groups of the associated variational functional at $0$, obtaining a nontrivial solution using Morse theory. In the absence of a direct sum decomposition, we use a cohomological local splitting to get an estimate of the critical groups.

Our main result is for the $q$-superlinear case for the problem
\begin{equation}
\label{probbb} 
\tag{$P$}
\left\{\begin{aligned}
- \divg \left(|\nabla u|^{p-2}\, \nabla u + a(x)\, |\nabla u|^{q-2}\, \nabla u\right) & = \lambda\, |u|^{p-2}\, u + |u|^{r-2}\, u + h(x,u) && \text{in } \Omega\\
u & = 0 && \text{on } \bdry{\Omega},
\end{aligned}\right.
\end{equation}
where $\lambda \in \R$ is a parameter, $q<r<p^\ast$ and $h$ is a Carath\'{e}odory function on $\Omega \times \R$ satisfying
\begin{equation}
\label{hgrowth} 
|h(x,t)| \le C \left(|t|^{\rho - 1} + |t|^{\sigma - 1}\right), 
\quad \text{for a.a. } x \in \Omega \text{ and all } t \in \R,
\end{equation}
for some $p < \sigma < \rho < r$ and $C > 0$. The notion 
of weak solution for problem \eqref{probbb} is formulated in a suitable Orlicz Sobolev space $W^{1,\H}_0(\Omega)$ that will be introduced in Section~\ref{prelim}.
\vskip4pt
\noindent
Let $\seq{\lambda_k}\subset\R^+$ be the sequence of (variational) eigenvalues of the $p$-Laplacian operator defined via cohomological index,
cf.\ formula \eqref{23}. Let us set
 $$
 G(x,t):=\frac{|t|^r}{r}+H(x,t),\quad\,\,\, x\in\Omega,\,\,\, t\in\R,
 $$
 being $H(x,t)=\displaystyle\int_0^t h(x,\tau)d\tau$.
The following is the main result of the paper.

\begin{theorem}
	\label{main}
	Assume that conditions~\eqref{p-q-cc} and \eqref{hgrowth} hold.
	Then problem \eqref{probbb} has a nontrivial weak solution $u\in W^{1,\H}_0(\Omega)$ in each of these cases:
	\begin{enumarab}
		\item $\lambda \notin \{\lambda_k\}_ {k \ge 1}$;
		\item for some $\delta > 0$, $G(x,t) \le 0$ for a.a. $x \in \Omega$ and $|t| \le \delta$;
		\item $G(x,t) \ge c\, |t|^s$ for a.a. $x \in \Omega$ and all $t \in \R$ for some $s \in (p,q)$ and $c > 0$.
	\end{enumarab}
\end{theorem}

\noindent
To our knowledge this is the first existence result for double-phase problems 
\eqref{1} in the framework of Morse theory and it is obtained by analyzing 
the critical groups $H^q(\Phi^0 \cap U,\Phi^0 \cap U \setminus \set{0})$ of the associated 
energy functional $\Phi$ at zero, $q\in\N$.

\section{Preliminaries and proof}
\label{prelim}

\subsection{Variational setting}

The Musielak-Orlicz space $L^\H(\Omega)$ associated with the function
\[
\H : \Omega \times [0,\infty) \to [0,\infty), \quad (x,t) \mapsto t^p + a(x)\, t^q
\]
consists of all measurable functions $u : \Omega \to \R$ with the $\H$-modular
\[
\rho_\H(u) := \int_\Omega \H(x,|u|)\, dx < \infty,
\]
endowed with the Luxemburg norm
\[
\norm[\H]{u} := \inf \set{\gamma > 0 : \rho_\H(u/\gamma) \le 1}.
\]
The space $L^\H(\Omega)$ is a uniformly convex, and hence reflexive, Banach space. Denoting by $\norm[p]{\cdot}$ the norm in $L^p(\Omega)$ and by $L^q_a(\Omega)$ the space of all measurable functions $u : \Omega \to \R$ with the seminorm
\[
\norm[q,a]{u} := \left(\int_\Omega a(x)\, |u|^q\, dx\right)^{1/q} < \infty,
\]
we have the continuous embeddings 
$$
L^q(\Omega) \hookrightarrow L^\H(\Omega) \hookrightarrow L^p(\Omega) \cap L^q_a(\Omega),
$$ 
see \cite[Proposition 2.15 (i), (iv), (v)]{CoSq}. 
Since $\rho_\H(u/\norm[\H]{u}) = 1$ whenever $u \ne 0$, we have
\begin{equation} \label{4}
\min \set{\norm[\H]{u}^p,\norm[\H]{u}^q} \le \norm[p]{u}^p + \norm[q,a]{u}^q \le \max \set{\norm[\H]{u}^p,\norm[\H]{u}^q}, \quad \forall u \in L^\H(\Omega).
\end{equation}
The related Sobolev space $W^{1,\H}(\Omega)$ consists of all functions $u$ in $L^\H(\Omega)$ with $|\nabla u| \in L^\H(\Omega)$, normed by
\[
\norm[1,\H]{u} := \norm[\H]{u} + \norm[\H]{\nabla u},
\]
where $\norm[\H]{\nabla u} = \norm[\H]{|\nabla u|}$. The completion 
of $C^\infty_0(\Omega)$ in $W^{1,\H}(\Omega)$ is denoted by $W^{1,\H}_0(\Omega)$ and it can be equivalently renormed by
\[
\norm{u} := \norm[\H]{\nabla u}
\]
via a Poincar\'e-type inequality, cf.\ \cite[Proposition 2.18, (iv)]{CoSq}, under assumption~\eqref{p-q-cc}. The spaces $W^{1,\H}(\Omega)$ and $W^{1,\H}_0(\Omega)$ are uniformly convex, and hence reflexive, Banach spaces.\ The Sobolev embedding $W^{1,\H}_0(\Omega) \hookrightarrow L^r(\Omega)$ is compact since $r < p^\ast$,  cf.\ 
\cite[Proposition 2.15, (iii)]{CoSq}. We have
\begin{equation} \label{5}
\min \set{\norm{u}^p,\norm{u}^q} \le \norm[p]{\nabla u}^p + \norm[q,a]{\nabla u}^q \le \max \set{\norm{u}^p,\norm{u}^q}, \quad \forall u \in W^{1,\H}_0(\Omega),
\end{equation}
by virtue of \eqref{4}.
A weak solution of problem \eqref{1} is a function $u \in W^{1,\H}_0(\Omega)$ satisfying
\[
\int_\Omega \left(|\nabla u|^{p-2} + a(x)\, |\nabla u|^{q-2}\right) \nabla u \cdot \nabla v\, dx = \int_\Omega f(x,u)\, v\, dx, 
\quad \forall v \in W^{1,\H}_0(\Omega).
\]
Weak solutions coincide with critical points of the functional
\[
\Phi(u) = \int_\Omega \left[\frac{1}{p}\, |\nabla u|^p + \frac{a(x)}{q}\, |\nabla u|^q - F(x,u)\right] dx, \quad u \in W^{1,\H}_0(\Omega),
\]
where $F(x,t) = \displaystyle\int_0^t f(x,s)\, ds$, by the following proposition.

\begin{proposition}[$C^1$ energy]
	Assume that \eqref{2} holds. Then $\Phi$ is of class $C^1$ with
	\begin{equation} 
		\label{3}
	\langle \Phi'(u), v\rangle = \int_\Omega \left[\left(|\nabla u|^{p-2} + a(x)\, |\nabla u|^{q-2}\right) \nabla u \cdot \nabla v - f(x,u)\, v\right] dx,
	\end{equation}
	for every $u, v \in W^{1,\H}_0(\Omega)$.
\end{proposition}

\begin{proof}
	In view of the embeddings mentioned above, \eqref{3} is clear. To see that $\Phi'$ is continuous, suppose that $u_j \to u$ in $W^{1,\H}_0(\Omega)$. For all $v \in W^{1,\H}_0(\Omega)$ with $\norm{v} = 1$, by the H\"{o}lder inequality,
	\begin{multline*}
	|\langle \Phi'(u_j) - \Phi'(u), v\rangle| \le \norm[p']{\abs{|\nabla u_j|^{p-2}\, \nabla u_j - |\nabla u|^{p-2}\, \nabla u}} \norm[p]{\nabla v}\\[5pt]
	+ \norm[q']{a(x)^{1/q'} \abs{|\nabla u_j|^{q-2}\, \nabla u_j - |\nabla u|^{q-2}\, \nabla u}} \norm[q,a]{\nabla v} + \norm[r']{f(x,u_j) - f(x,u)} \norm[r]{v},
	\end{multline*}
	where $s' = s/(s - 1)$ is the H\"{o}lder conjugate of $s$. Since $L^\H(\Omega) \hookrightarrow L^p(\Omega) \cap L^q_a(\Omega)$ and $W^{1,\H}_0(\Omega) \hookrightarrow L^r(\Omega)$, $\nabla u_j \to \nabla u$ in $L^p(\Omega) \cap L^q_a(\Omega)$, $u_j \to u$ in $L^r(\Omega)$, and $\norm[p]{\nabla v}$, $\norm[q,a]{\nabla v}$ and $\norm[r]{v}$ are uniformly bounded, the assertion follows from the dominated convergence theorem and \eqref{2}.
\end{proof}

\subsection{Palais-Smale condition}

The operator $A : W^{1,\H}_0(\Omega) \to (W^{1,\H}_0(\Omega))'$ defined by
\[
\langle A(u), v\rangle := \int_\Omega \left[\left(|\nabla u|^{p-2} + a(x)\, |\nabla u|^{q-2}\right) \nabla u \cdot \nabla v\right] dx, \quad u, v \in W^{1,\H}_0(\Omega),
\]
where $(W^{1,\H}_0(\Omega))'$ is the dual space of $W^{1,\H}_0(\Omega)$, has the following important property.

\begin{proposition} \label{Proposition 1}
	If $u_j \wto u$ in $W^{1,\H}_0(\Omega)$ and $A(u_j)\, (u_j - u) \to 0$, then $u_j \to u$ in $W^{1,\H}_0(\Omega)$.
\end{proposition}

\begin{proof}
	Noting that
	\[
	\langle A(u), v\rangle \le \norm[p]{\nabla u}^{p-1} \norm[p]{\nabla v} + \norm[q,a]{\nabla u}^{q-1} \norm[q,a]{\nabla v} \quad \forall u, v \in W^{1,\H}_0(\Omega)
	\]
	by the H\"{o}lder inequality, and the equality holds when $u = v$, we have
	\begin{multline*}
	0 \le \big(\norm[p]{\nabla u_j}^{p-1} - \norm[p]{\nabla u}^{p-1}\big) \! \big(\norm[p]{\nabla u_j} - \norm[p]{\nabla u}\big)\\[5pt]
	+ \big(\norm[q,a]{\nabla u_j}^{q-1} - \norm[q,a]{\nabla u}^{q-1}\big) \! \big(\norm[q,a]{\nabla u_j} - \norm[q,a]{\nabla u}\big) \le \langle A(u_j) - A(u), u_j - u \rangle\to 0,
	\end{multline*}
	so that $\norm[p]{\nabla u_j} \to \norm[p]{\nabla u}$ and $\norm[q,a]{\nabla u_j} \to \norm[q,a]{\nabla u}$. Then $\nabla u_j \to \nabla u$ in $L^p(\Omega) \cap L^q_a(\Omega)$ by uniform convexity, and hence the conclusion follows from from \eqref{5}.
\end{proof}

Recall that the functional $\Phi$ satisfies the Palais-Smale compactness condition at the level $c \in \R$, or \PS{c} for short, if every sequence $\seq{u_j} \subset W^{1,\H}_0(\Omega)$ such that $\Phi(u_j) \to c$ and $\Phi'(u_j) \to 0$, called a \PS{c} sequence, has a convergent subsequence. We say that $\Phi$ satisfies the \PS{} condition if it satisfies the \PS{c} condition for all $c \in \R$. When verifying these conditions, it suffices to show that $\seq{u_j}$ is bounded by the following proposition.

\begin{proposition}[Bounded Palais-Smale condition] 
	\label{Proposition 4}
	Every bounded sequence $\seq{u_j} \subset W^{1,\H}_0(\Omega)$ such that $\Phi'(u_j) \to 0$ has a convergent subsequence.
\end{proposition}

\begin{proof}
	Since $\seq{u_j}$ is bounded, a renamed subsequence converges to some $u$ weakly in $W^{1,\H}_0(\Omega)$ and strongly in $L^r(\Omega)$. Then
	\[
	\langle A(u_j), u_j - u\rangle = \langle \Phi'(u_j), u_j - u\rangle + \int_\Omega f(x,u_j)\, (u_j - u)\, dx \to 0
	\]
	since
	\[
	\abs{\int_\Omega f(x,u_j)\, (u_j - u)\, dx} \le C \left(\norm[r]{u_j}^{r-1} + 1\right) \norm[r]{u_j - u}
	\]
	by \eqref{2} and the H\"{o}lder inequality, so the conclusion follows from Proposition \ref{Proposition 1}.
\end{proof}

\subsection{Regularity estimates}

For $f \in L^m(\Omega)$ with $m >1$, solutions of
\begin{equation} \label{10}
\int_\Omega \left(|\nabla u|^{p-2} + a(x)\, |\nabla u|^{q-2}\right) \nabla u \cdot \nabla v\, dx = \int_{\Omega} f(x)\, v\, dx \quad \forall v \in W^{1,\H}_0(\Omega)
\end{equation}
enjoy the natural $L^m$-estimates given in the following proposition.

\begin{proposition} \label{Proposition 5}
	Let $f \in L^m(\Omega),\, 1 < m \le \infty$ and let $u \in W^{1,\H}_0(\Omega)$ satisfy \eqref{10}. Then
	\begin{equation} \label{11}
	\norm[r]{u} \le C \norm[m]{f}^{1/(p-1)},
	\end{equation}
	where we have set
	\[
	r = \begin{cases}
	\dfrac{N\, (p - 1)\, m}{N - pm}, & 1< m < \dfrac{N}{p}\\[10pt]
	\infty, & m > \dfrac{N}{p}
	\end{cases}
	\]
	and $C = C(N,\Omega,p,m) > 0$.
\end{proposition}

\begin{proof}
	For $k, \alpha > 0$ and $t \in \R$, set $t_k = \max \set{-k,\min \set{t,k}}$ and consider 
	the nondecreasing function $g(t) = t_k^\alpha$ (with the agreement $a^\alpha:=|a|^{\alpha-1}a,$ for $a\in\R$). 
	Testing equation \eqref{10} with the $g(u) \in W^{1,\H}_0(\Omega)$ provides the inequality
	\[
	\norm[p]{\nabla G(u)}^p \le \int_\Omega f(x)\, g(u)\, dx,
	\]
	where
	\[
	G(t) := \int_0^t g'(s)^{1/p}\, ds = \frac{\alpha^{1/p}\, p}{\alpha + p - 1}\, t_k^{(\alpha + p - 1)/p},\quad t\in\R.
	\]
	Using the Sobolev inequality on the left and the H\"{o}lder inequality on the right now gives
	\begin{equation} \label{12}
	\big\|u_k^{(\alpha + p - 1)/p}\big\|_{p^\ast}^p \le C \norm[m]{f} \norm[m']{u_k^\alpha}.
	\end{equation}
	If $1< m < N/p$, take
	\[
	\alpha = \frac{(p - 1)\, p^\ast}{pm' - p^\ast} = \frac{N\, (p - 1)\, (m - 1)}{N - pm} > 0,
	\]
	so that
	\[
	\frac{(\alpha + p - 1)\, p^\ast}{p} = \alpha m' =: r.
	\]
	Then $r = N\, (p - 1)\, m/(N - pm)$ and \eqref{12} gives $\norm[r]{u_k}^{pr/p^\ast} \le C \norm[m]{f} \norm[r]{u_k}^{r/m'}$, so
	\[
	\norm[r]{u_k} \le C \norm[m]{f}^{1/(p-1)}.
	\]
	Letting $k \to \infty$ gives \eqref{11} for this case. If $N/p < m \le \infty$, arguing as in \cite[Section 3.2]{CoSq} gives
	\begin{equation} \label{13}
	\norm[\infty]{u} \le C\,  \norm[m]{f}^{1/(p-1)}.
	\end{equation}
This concludes the proof.
\end{proof}


\subsection{Critical groups at zero}

In this subsection we consider the problem
\begin{equation} \label{6}
\left\{\begin{aligned}
- \divg \left(|\nabla u|^{p-2}\, \nabla u + a(x)\, |\nabla u|^{q-2}\, \nabla u\right) & = \lambda\, |u|^{p-2}\, u + g(x,u) && \text{in } \Omega\\
u & = 0 && \text{on } \bdry{\Omega},
\end{aligned}\right.
\end{equation}
where $\lambda \in \R$ is a parameter and $g$ is a Carath\'{e}odory function on $\Omega \times \R$ satisfying
\begin{equation} \label{7}
|g(x,t)| \le C \left(|t|^{r-1} + |t|^{\sigma - 1}\right) \quad \text{for a.a. } x \in \Omega \text{ and all } t \in \R
\end{equation}
for some $p < \sigma < r < p^\ast$ and $C > 0$. Problem \eqref{6} has the trivial solution $u = 0$, and we study the critical groups at $0$ of the associated functional
\[
\Phi(u) = \int_\Omega \left[\frac{1}{p}\, |\nabla u|^p + \frac{a(x)}{q}\, |\nabla u|^q - \frac{\lambda}{p}\, |u|^p - G(x,u)\right] dx, \quad u \in W^{1,\H}_0(\Omega),
\]
where $G(x,t) = \int_0^t g(x,s)\, ds$.
Let us recall that the critical groups of $\Phi$ at $0$ are given by
\begin{equation} \label{24}
C^q(\Phi,0) := H^q(\Phi^0 \cap U,\Phi^0 \cap U \setminus \set{0}), \quad q \in\N,
\end{equation}
where $\Phi^0 = \bgset{u \in W^{1,\H}_0(\Omega) : \Phi(u) \le 0}$, $U$ is any neighborhood of $0$, and $H$ denotes Alexander-Spanier cohomology with $\Z_2$-coefficients. They are independent of $U$ by the excision property of the cohomology groups. They are also invariant under homotopies that preserve the isolatedness of the critical point by the following proposition (see Chang and Ghoussoub \cite{MR1422006} or Corvellec and Hantoute \cite{MR1926378}).

\begin{proposition}[Homotopical invariance]
	\label{Proposition 2}
	Let $\Phi_\tau,\, \tau \in [0,1]$ be a family of $C^1$-functionals on a Banach space $W$ such that $0$ is a critical point of each $\Phi_\tau$. If there is a closed neighborhood $U$ of $0$ such that
	\begin{enumarab}
		\item each $\Phi_\tau$ satisfies the {\em \PS{}} condition over $U$,
		\item $U$ contains no other critical point of any $\Phi_\tau$,
		\item the map $[0,1] \to C^1(U,\R),\, \tau \mapsto \Phi_\tau$ is continuous,
	\end{enumarab}
	then $C^q(\Phi_0,0) \isom C^q(\Phi_1,0)$ for all $q$.
\end{proposition}

\noindent
First we show that the critical groups of $\Phi$ at $0$ depend only on the values of $g(x,t)$ for small $|t|$.

\begin{lemma} \label{Lemma 1}
	Let $\delta > 0$ and let $\vartheta : \R \to [- \delta,\delta]$ be a smooth nondecreasing function such that 
	$$
	\vartheta(t) = - \delta\quad\text{for $t \le - \delta$},\,\quad 
	\vartheta(t) = t\quad\text{for $- \delta/2 \le t \le \delta/2$},\,\quad 
	\vartheta(t) = \delta\quad\text{for $t \ge \delta$}. 
	$$
	Let us set
	\[
	\Phi_1(u) = \int_\Omega \left[\frac{1}{p}\, |\nabla u|^p + \frac{a(x)}{q}\, |\nabla u|^q - \frac{\lambda}{p}\, |u|^p - G(x,\vartheta(u))\right] dx, \quad u \in W^{1,\H}_0(\Omega).
	\]
	If $0$ is an isolated critical point of $\Phi$, then it is also an isolated critical point of $\Phi_1$ and 
	$$
	C^q(\Phi,0) \isom C^q(\Phi_1,0),\quad\text{for all $q$.}
	$$
\end{lemma}

\begin{proof}
	We apply Proposition \ref{Proposition 2} to the family of functionals, for $u \in W^{1,\H}_0(\Omega)$ and $\tau \in [0,1]$,
	\begin{equation*}
	\Phi_\tau(u) := \int_\Omega \left[\frac{1}{p}\, |\nabla u|^p + \frac{a(x)}{q}\, |\nabla u|^q - \frac{\lambda}{p}\, |u|^p - G(x,(1 - \tau)\, u + \tau\, \vartheta(u))\right] dx,
	\end{equation*}
	in a ball $B_\varepsilon(0) = \{u \in W^{1,\H}_0(\Omega) : \norm{u} \le \varepsilon\}$ for $\varepsilon>0$ small, after noting that $\Phi_0 = \Phi$. Proposition \ref{Proposition 4} implies that each $\Phi_\tau$ satisfies the Palais-Smale condition over the ball $B_\varepsilon(0)$ and it is readily seen that the map 
	$[0,1]\ni\tau\mapsto  \Phi_\tau\in  C^1(B_\varepsilon(0),\R)$ is continuous, so it only remains to show that for sufficiently small $\varepsilon > 0$, $B_\varepsilon(0)$ contains no critical point of any $\Phi_\tau$ other than $0$.
	Suppose $u_j \to 0$ in $W^{1,\H}_0(\Omega)$, $\Phi_{\tau_j}'(u_j) = 0,\, \tau_j \in [0,1]$ and $u_j \ne 0$. Then $u_j$ is a weak solution to 
	\[
	\left\{\begin{aligned}
	- \divg \left(|\nabla u_j|^{p-2}\, \nabla u_j + a(x)\, |\nabla u_j|^{q-2}\, \nabla u_j\right) & = \lambda\, |u_j|^{p-2}\, u_j + g_j(x,u_j) && \text{in } \Omega\\
	u_j & = 0 && \text{on } \bdry{\Omega},
	\end{aligned}\right.
	\]
	where we have set
	\[
	g_j(x,t) = (1 - \tau_j + \tau_j\, \vartheta'(t))\, g(x,(1 - \tau_j)\, t + \tau_j\, \vartheta(t)).
	\]
	Since $(1 - \tau_j)\, t + \tau_j\, \vartheta(t) = t$ for $|t| \le \delta/2$ and $|(1 - \tau_j)\, t + \tau_j\, \vartheta(t)| \le |t| + \delta < 3\, |t|$ for $|t| > \delta/2$, the growth estimate \eqref{7} implies that, for some $C>0$ independent of $j$,
	\[
	|g_j(x,t)| \le C \left(|t|^{r-1} + |t|^{\sigma - 1}\right) \quad \text{for a.a. } x \in \Omega \text{ and all } t \in \R.
	\]
	Then  $u_j \in L^\infty(\Omega)$ (cf.\ \cite[Section 3.2]{CoSq}) with $L^\infty$-bound independent of $j$. 
	Since $u_j \to 0$ in $W^{1,p}_0(\Omega)$, it follows $\|u_j\|_\ell\to 0$ for any $\ell\geq 1$, as $j\to\infty$.
	By Proposition \ref{Proposition 5}, applied with the choice
	$$
	f_j(x):=\lambda |u_j(x)|^{p-2}u_j(x)+g_j(x,u_j(x)),\quad j\in\N,\,\, x\in\Omega,
	$$
	we get $\|u_j\|_{\infty} \to 0$ since for a fixed $m_0>N/p$ we have, for $j\to\infty$,
	$$
	\int_\Omega |f_j|^{m_0} dx\leq C\|u_j\|_{m_0(p-1)}^{m_0(p-1)}+C\|u_j\|_{m_0(r-1)}^{m_0(r-1)}+C\|u_j\|_{m_0(\sigma-1)}^{m_0(\sigma-1)}\to 0.
	$$
	For sufficiently large $j$ we thus have  $|u_j(x)| \le \delta/2$ for a.e.\ $x\in\Omega$
	and, hence, $\Phi'(u_j) = \Phi_{\tau_j}'(u_j) = 0$, contradicting the	
	assumption that $0$ is an {\em isolated} critical point of $\Phi$.
\end{proof}

\noindent
In the absence of a direct sum decomposition, the main technical tool to get an estimate of the critical groups is the notion of cohomological local splitting introduced in Perera, Agarwal and O'Regan \cite{MR2640827}, which is a variant of the homological linking of Perera \cite{MR1700283} (see 
\cite{MR1312028}). The following slightly different form of this notion was given in Degiovanni, Lancelotti, and Perera \cite{MR2661274}.

\begin{definition}
	We say that a $C^1$-functional $\Phi$ on a Banach space $W$ has a cohomological local splitting near $0$ in dimension $k \ge 1$ if there are symmetric cones $W_\pm \subset W$ with $W_+ \cap W_- = \set{0}$ and $\rho > 0$ such that
	\[
	i(W \setminus W_+) = i(W_- \setminus \set{0}) = k
	\]
	and
	\begin{equation} \label{16}
	\Phi(u) \ge \Phi(0) \quad \forall u \in B_\rho \cap W_+, \qquad \Phi(u) \le \Phi(0) \quad \forall u \in B_\rho \cap W_-,
	\end{equation}
	where $i$ denotes the $\Z_2$-cohomological index and $B_\rho = \set{u \in W : \norm{u} \le \rho}$.
\end{definition}

We recall the definition of the cohomological index (see Fadell and Rabinowitz \cite{MR57:17677}). For a symmetric subset $M$ of $W \setminus \set{0}$, let $\overline{M} = M/\Z_2$ be the quotient space of $M$ with each $u$ and $-u$ identified, let $f : \overline{M} \to \RP^\infty$ be the classifying map of $\overline{M}$, and let $f^\ast : H^\ast(\RP^\infty) \to H^\ast(\overline{M})$ be the induced homomorphism of the Alexander-Spanier cohomology rings. Then the cohomological index of $M$ is defined by
\[
i(M) = \begin{cases}
\sup\, \set{m \ge 1 : f^\ast(\omega^{m-1}) \ne 0}, & M \ne \emptyset\\[5pt]
0, & M = \emptyset,
\end{cases}
\]
where $\omega \in H^1(\RP^\infty)$ is the generator of the polynomial ring $H^\ast(\RP^\infty) = \Z_2[\omega]$. For example, the classifying map of the unit sphere $S^{m-1}$ in $\R^m,\, m \ge 1$ is the inclusion $\RP^{m-1} \incl \RP^\infty$, which induces isomorphisms on $H^q$ for $q \le m-1$, so $i(S^{m-1}) = m$.

\begin{proposition}[{Degiovanni, Lancelotti, and Perera \cite[Proposition 2.1]{MR2661274}}] \label{Proposition 3}
	Assume that $0$ is an isolated critical point of $\Phi$ and that $\Phi$ has a cohomological local splitting near $0$ in dimension $k$.
	Then it holds $C^k(\Phi,0) \ne 0.$
\end{proposition}

\noindent
In order to give sufficient conditions for $\Phi$ to have a cohomological local splitting near $0$, and hence a nontrivial critical group by Proposition \ref{Proposition 3}, consider the asymptotic eigenvalue problem
\begin{equation} \label{8}
\left\{\begin{aligned}
- \Delta_p\, u & = \lambda\, |u|^{p-2}\, u && \text{in } \Omega\\[1pt]
u & = 0 && \text{on } \bdry{\Omega}.
\end{aligned}\right.
\end{equation}
Let
\[
I(u) = \int_\Omega |\nabla u|^p\, dx, \quad J(u) = \int_\Omega |u|^p\, dx, \quad u \in W^{1,p}_0(\Omega),
\]
and set
\[
\Psi(u) = \frac{1}{J(u)}, \quad u \in \M = \big\{u \in W^{1,p}_0(\Omega) : I(u) = 1\big\}.
\]
Then eigenvalues of problem \eqref{8} coincide with critical values of $\Psi$. Let $\F$ denote the class of symmetric subsets of $\M$, and set
\begin{equation} \label{23}
\lambda_k := \inf_{\substack{M \in \F\\[1pt]
		i(M) \ge k}}\, \sup_{u \in M}\, \Psi(u), \quad k \ge 1.
\end{equation}
Then $0 < \lambda_1 < \lambda_2 \le \lambda_3 \le \cdots \nearrow + \infty$ is a sequence of eigenvalues of \eqref{8} and
\begin{equation} \label{9}
\lambda_k < \lambda_{k+1} \implies i(\set{u \in \M : \Psi(u) \le \lambda_k}) = i(\set{u \in \M : \Psi(u) < \lambda_{k+1}}) = k
\end{equation}
(see \cite[Propositions 3.52 and 3.53]{MR2640827}). The main result of this subsection is the following theorem.

\begin{theorem}[Critical groups at $0$] 
	\label{Theorem 1}
	Assume that $g$ satisfies \eqref{7} and $0$ is an isolated critical point of $\Phi$.
	\begin{enumarab}
		\item $C^0(\Phi,0) \isom \Z_2$ and $C^q(\Phi,0) = 0$ for $q \ge 1$ in the following cases:
		\begin{enumroman}
			\item $\lambda < \lambda_1$;
			\item $\lambda = \lambda_1$ and, for some $\delta > 0$, $G(x,t) \le 0$ for a.a. $x \in \Omega$ and $|t| \le \delta$.
		\end{enumroman}
		\item $C^k(\Phi,0) \ne 0$ in the following cases:
		\begin{enumroman}
			\item $\lambda_k < \lambda < \lambda_{k+1}$;
			\item $\lambda_k < \lambda = \lambda_{k+1}$ and, for some $\delta > 0$, $G(x,t) \le 0$ for a.a. $x \in \Omega$ and $|t| \le \delta$;
			\item $\lambda_k = \lambda < \lambda_{k+1}$ and $G(x,t) \ge c\, |t|^s$ for a.a. $x \in \Omega$ and all $t \in \R$ for some $s \in (p,q)$ and $c > 0$.
		\end{enumroman}
	\end{enumarab}
\end{theorem}

\begin{proof}
	We have
	\begin{equation} \label{14}
	\Phi(u) = \frac{1}{p}\, \big[I(u) - \lambda\, J(u)\big] + \int_\Omega \left[\frac{a(x)}{q}\, |\nabla u|^q - G(x,u)\right] dx.
	\end{equation}
	By \eqref{7} and the Sobolev embedding, we have 
	\begin{equation} \label{19}
	\int_\Omega G(x,u)\, dx = \o\big(\norm[p]{\nabla u}^p\big), \quad \text{as } \norm[p]{\nabla u} \to 0,
	\end{equation}
	and in view of Lemma \ref{Lemma 1}, without loss of generality, 
	we may assume that the sign conditions on $G$ in (1)({\em b}) and (2)({\em b}) hold for every $t \in \R$.
\vskip4pt
\noindent	
{\rm (1)} We show that $0$ is a local minimizer of $\Phi$. Since $\Psi(u) \ge \lambda_1$ for all $u \in \M$, we have
	\begin{equation} \label{15}
	I(u) \ge \lambda_1\, J(u), \quad \forall u \in W^{1,p}_0(\Omega).
	\end{equation}
\noindent	
	$\bullet$ ({\em a}) By \eqref{14}--\eqref{15}, we get
	\[
	\Phi(u) \ge \frac{1}{p} \left(1 - \frac{\lambda_+}{\lambda_1} + \o(1)\right) \norm[p]{\nabla u}^p, \quad \text{as } \norm[p]{\nabla u} \to 0,
	\]
	where $\lambda_+ = \max \set{\lambda,0}$. So $\Phi(u) \ge 0$ for all $u \in B_\rho$ for sufficiently small $\rho > 0$ by \eqref{5}.
	\vskip4pt
	\noindent
	$\bullet $ ({\em b}) By \eqref{14} and \eqref{15}, we get
	\[
	\Phi(u) \ge - \int_\Omega G(x,u)\, dx \ge 0, \quad \forall u \in W^{1,\H}_0(\Omega).
	\]
	\vskip4pt
	\noindent	
	{\rm (2)} We show that $\Phi$ has a cohomological local splitting near $0$ in dimension $k$ and then apply Proposition \ref{Proposition 3}.\ 
	In light of  \cite[Theorem 2.3]{MR2514055}, the set $\{u \in W^{1,p}_0(\Omega) : I(u) \le \lambda_k\, J(u)\}$ contains a {\em symmetric cone} $W_-$ with
	$i(W_- \setminus \set{0}) = k$
	and $\{u \in W_- : \norm[p]{u} = 1\}$ is bounded in $C^1(\Omega)$, so that, in particular, we have the inequality
	\begin{equation} \label{20}
	\int_\Omega \frac{a(x)}{q}\, |\nabla u|^q\, dx \le C \norm[p]{u}^q, \quad\,\,\, \forall u \in W_-,
	\end{equation}
	for some $C > 0$. Since $W^{1,\H}_0(\Omega)$ is embedded in $W^{1,p}_0(\Omega)$ as a {\em dense linear subspace}, the inclusion 
	$$
	\big\{u \in W^{1,\H}_0(\Omega) : I(u) < \lambda_{k+1}\, J(u)\big\} \incl \big\{u \in W^{1,p}_0(\Omega) : I(u) < \lambda_{k+1}\, J(u)\big\}
	$$ 
	is a {\em homotopy equivalence} by Palais (cf.\ \cite[Theorem 17]{MR0189028}), so
	\[
	i\big(\big\{u \in W^{1,\H}_0(\Omega) : I(u) < \lambda_{k+1}\, J(u)\big\}\big) = 
	i\big(\big\{u \in W^{1,p}_0(\Omega) : I(u) < \lambda_{k+1}\, J(u)\big\}\big) = k
	\]
	by virtue of \eqref{9}. We take now
	\[
	W_+ := \big\{u \in W^{1,\H}_0(\Omega) : I(u) \ge \lambda_{k+1}\, J(u)\big\}.
	\]
	It only remains to show that \eqref{16} holds for sufficiently small $\rho > 0$.
	\vskip4pt
	\noindent
	$\bullet$ ({\em a}) For $u \in W_+$, we obtain
	\[
	\Phi(u) \ge \frac{1}{p} \left(1 - \frac{\lambda}{\lambda_{k+1}} + \o(1)\right) \norm[p]{\nabla u}^p, \quad \text{as } \norm[p]{\nabla u} \to 0
	\]
	by virtue of \eqref{14} and \eqref{19}. So $\Phi(u) \ge 0$ for all $u \in B_\rho \cap W_+$ for sufficiently small $\rho > 0$ by \eqref{5}. For $u \in W_-$,
	\[
	\Phi(u) \le - \frac{1}{p} \left(\frac{\lambda}{\lambda_k} - 1 + \o(1)\right) \norm[p]{\nabla u}^p \quad \text{as } \norm[p]{\nabla u} \to 0
	\]
	by \eqref{14}, \eqref{19}, and \eqref{20} since $q > p$. So $\Phi(u) \le 0$ for all $u \in B_\rho \cap W_-$ for small $\rho > 0$ by \eqref{5}.
	\vskip2pt
	\noindent
	$\bullet$ ({\em b}) For $u \in W_+$, we have
	\[
	\Phi(u) \ge - \int_\Omega G(x,u)\, dx \ge 0
	\]
	by \eqref{14}, and $\Phi(u) \le 0$ for all $u \in B_\rho \cap W_-$ for small $\rho > 0$ as in ({\em a}).
	\vskip2pt
	\noindent
	$\bullet$ ({\em c}) We have $\Phi(u) \ge 0$ for all $u \in B_\rho \cap W_+$ for sufficiently small $\rho > 0$ as in ({\em i}). For $u \in W_-$,
	\[
	\Phi(u) \le C \norm[p]{u}^q - \frac{\norm[p]{u}^s}{C}
	\]
	for some $C > 0$ by \eqref{14}, \eqref{20}, and since $s > p$. Since $s < q$, then $\Phi(u) \le 0$ for all $u \in B_\rho \cap W_-$ for sufficiently small $\rho > 0$ by \eqref{5}.
\end{proof}

\subsection{Nontrivial solutions}

In this subsection we obtain a nontrivial solution of the problem
\begin{equation} \label{17}
\left\{\begin{aligned}
- \divg \left(|\nabla u|^{p-2}\, \nabla u + a(x)\, |\nabla u|^{q-2}\, \nabla u\right) & = \lambda\, |u|^{p-2}\, u + |u|^{r-2}\, u + h(x,u) && \text{in } \Omega\\
u & = 0 && \text{on } \bdry{\Omega},
\end{aligned}\right.
\end{equation}
where $\lambda \in \R$ is a parameter, $r \in (q,p^\ast)$, and $h$ is a Carath\'{e}odory function on $\Omega \times \R$ satisfying
\begin{equation} \label{18}
|h(x,t)| \le C \left(|t|^{\rho - 1} + |t|^{\sigma - 1}\right) \quad \text{for a.a. } x \in \Omega \text{ and all } t \in \R
\end{equation}
for some $p < \sigma < \rho < r$ and $C > 0$.
First we verify that the associated functional
\[
\Phi(u) = \int_\Omega \left[\frac{1}{p}\, |\nabla u|^p + \frac{a(x)}{q}\, |\nabla u|^q - \frac{\lambda}{p}\, |u|^p - \frac{1}{r}\, |u|^r - H(x,u)\right] dx, \quad u \in W^{1,\H}_0(\Omega),
\]
where $H(x,t) =\int_0^t h(x,s)\, ds$, satisfies the \PS{} condition. We note that
\begin{multline} \label{21}
q\, \Phi(u) - \langle \Phi'(u), u\rangle = \left(\frac{q}{p} - 1\right) \int_\Omega \big(|\nabla u|^p - \lambda\, |u|^p\big)\, dx + \left(1 - \frac{q}{r}\right) \int_\Omega |u|^r\, dx\\[5pt]
+ \int_\Omega \big(h(x,u)\, u - q\, H(x,u)\big)\, dx.
\end{multline}

\begin{lemma}[Palais-Smale condition] 
	\label{Lemma 3}
	Every sequence $\seq{u_j} \subset W^{1,\H}_0(\Omega)$ such that $\seq{\Phi(u_j)}$ is bounded and $\Phi'(u_j) \to 0$ has a convergent subsequence.
\end{lemma}

\begin{proof}
	It suffices to show that $\seq{u_j}$ is bounded by Proposition \ref{Proposition 4}. Since $p < q < r$ and $\sigma < \rho < r$, it follows from \eqref{21}, \eqref{18} and the H\"{o}lder and Young's inequalities that $\norm[r]{u_j}^r \le C + \o(\norm{u_j})$ for some $C > 0$. Then
	\begin{equation*}
	\int_\Omega \left[\frac{1}{p}\, |\nabla u_j|^p + \frac{a(x)}{q}\, |\nabla u_j|^q\right] dx = \Phi(u_j) + \int_\Omega \left[\frac{\lambda}{p}\, |u_j|^p + \frac{1}{r}\, |u_j|^r + H(x,u_j)\right] dx
	\le C + \o(\norm{u_j}),
	\end{equation*}
	which together with \eqref{5} gives the desired conclusion.
\end{proof}

Next we study the structure of the sublevel sets of $\Phi$ at infinity.

\begin{lemma} \label{Lemma 2}
	There exists $\alpha < 0$ such that the sublevel set 
	$$
	\Phi^\alpha := \bgset{u \in W^{1,\H}_0(\Omega) : \Phi(u) \le \alpha}
	$$
	 is contractible in itself.
\end{lemma}

\begin{proof}
	Since $p < q < r$ and $\sigma < \rho < r$, it follows from \eqref{21}, \eqref{18}, and the Young's inequality that $\langle\Phi'(u), u\rangle - q \Phi(u)$ is bounded from above, so for $\alpha < 0$ with $|\alpha|$ sufficiently large,
	\begin{equation} \label{22}
	\langle \Phi'(u), u\rangle  < 0, \quad \forall u \in \Phi^\alpha.
	\end{equation}
	For $u \in W^{1,\H}_0(\Omega) \setminus \set{0}$, taking into account that $\Phi(tu) \to - \infty$ as $t \to + \infty$, set
	\[
	t(u) = \inf \set{t \ge 1 : \Phi(tu) \le \alpha},
	\]
	and note that the function $u \mapsto t(u)$ is continuous by \eqref{22} and the implicit function theorem. Then the map $u \mapsto t(u)\, u$ is a {\em retraction} of $W^{1,\H}_0(\Omega) \setminus \set{0}$ onto $\Phi^\alpha$, and the conclusion follows since the former is contractible in itself.
\end{proof}

\subsection{Proof of Theorem~\ref{main} completed}
We are now ready to prove the main result. Let $\seq{\lambda_k}$ be the sequence of eigenvalues of problem \eqref{8} defined in \eqref{23}.
%
%
	Suppose that $0$ is the only critical point of $\Phi$. Taking $U = W^{1,\H}_0(\Omega)$ in \eqref{24}, we have
	\[
	C^q(\Phi,0) = H^q(\Phi^0,\Phi^0 \setminus \set{0}).
	\]
	Let $\alpha < 0$ be as in Lemma \ref{Lemma 2}. Since $\Phi$ has no other critical points and satisfies the \PS{} condition by Lemma \ref{Lemma 3}, $\Phi^0$ is a deformation {\em retract} of $W^{1,\H}_0(\Omega)$ and $\Phi^\alpha$ is a 
	deformation {\em retract} of $\Phi^0 \setminus \set{0}$ by the second deformation lemma. So
	\[
	C^q(\Phi,0) \isom H^q(W^{1,\H}_0(\Omega),\Phi^\alpha) = 0 \quad \forall q\in\N,
	\]
	since $\Phi^\alpha$ is contractible in itself, contradicting Theorem \ref{Theorem 1} in each of the cases (1)--(3).
\qed

\def\cdprime{$''$}

\bigskip


\bigskip
\begin{thebibliography}{10}
	
	\bibitem{BarColMin1}
	P. Baroni, M. Colombo, G. Mingione, 
	{\em Harnack inequalities for double phase functionals}, 
	Nonlinear Anal.\ TMA (Special issue for Enzo Mitidieri) {\bf 121} (2015), 206--222. 
	
	\bibitem{BarColMin2}
	P. Baroni, M. Colombo, G. Mingione, 
	{\em Non-autonomous functionals, borderline cases and related function classes}, 
	St.\ Petersburg Math.\ J.\  {\bf 27} (2016), 347--379.
	
	\bibitem{BarKuuMin}
	P. Baroni, T. Kuusi, G. Mingione, 
	{\em Borderline gradient continuity of minima}, J. Fixed Point Theory Appl. {\bf 15} (2014), 537--575. 
	
	
	\bibitem{MR1422006}
	K.~C. Chang, N.~Ghoussoub,
	{\em The {C}onley index and the critical groups via an extension of
	{G}romoll-{M}eyer theory},
	Topol. Methods Nonlinear Anal. {\bf 7} (1996), 77--93.
	
	\bibitem{CoSq}
	F.\ Colasuonno, M.\ Squassina,
	{\em Eigenvalues for double phase variational integrals},
	Ann. Mat. Pura Appl., to appear,
	\href{doi:10.1007/s10231-015-0542-7}{\tt DOI link}.
	
	\bibitem{ColMing2}
	M. Colombo, G. Mingione, 
	\emph{Bounded minimisers of double phase variational integrals}, 
	Arch. Ration. Mech. Anal. {\bf 218} (2015), 219--273.
	
	\bibitem{ColMing1}
	M. Colombo, G. Mingione, \emph{Regularity for double phase variational problems},  
	Arch. Ration. Mech. Anal. {\bf 215} (2015), 443--496.
	
	\bibitem{MR1926378}
	J.-N. Corvellec, A.~Hantoute,
	{\em Homotopical stability of isolated critical points of continuous functionals},
	Set-Valued Anal. {\bf 10} (2002), 143--164.
	
	\bibitem{cup5}
	G.\ Cupini, P.\ Marcellini, E.\ Mascolo, 
	\emph{Local boundedness of minimizers with limit growth conditions,}
	J.\ Optimization Th.\ Appl. {\bf 166} (2015), 1--22.
	
	\bibitem{MR2514055}
	M.\ Degiovanni, S.\ Lancelotti,
	{\em Linking solutions for $p$-Laplace equations with nonlinearity at
	critical growth},
	J. Funct. Anal.  {\bf 256} (2009), 3643--3659.
	
	\bibitem{MR2661274}
	M.\ Degiovanni, S.\ Lancelotti, K.\ Perera,
	{\em Nontrivial solutions of {$p$}-superlinear {$p$}-{L}aplacian problems
	via a cohomological local splitting},
	Commun. Contemp. Math. {\bf 12} (2010), 475--486.
	
	\bibitem{MR57:17677}
	E~R. Fadell, P.~H. Rabinowitz,
	{\em Generalized cohomological index theories for {L}ie group actions with
	an application to bifurcation questions for {H}amiltonian systems,}
	Invent. Math. {\bf 45} (1978), 139--174.
	
	\bibitem{MR1312028}
	S.\ Li, M. Willem,
	{\em Applications of local linking to critical point theory},
	J. Math. Anal. Appl. {\bf 189} (1995), 6--32.
	
	\bibitem{Marcellini89} 
	P.\ Marcellini,
	{\em Regularity of minimisers of integrals of the calculus of variations with
		non standard growth conditions}, 
	Arch. Ration. Mech. Anal. \textbf{105} (1989), 267--284.
	
	\bibitem{Marcellini91} 
	P.\ Marcellini,
	\emph{Regularity and existence of solutions of elliptic equations with $p$,$q$-growth conditions}, 
	J.\ Differential Equations \textbf{90} (1991), 1--30.
	
	\bibitem{MR0189028}
	R.~S. Palais,
	{\em Homotopy theory of infinite dimensional manifolds},
	Topology {\bf 5} (1966), 1--16.
	
	\bibitem{MR1700283}
	K.\ Perera, 
	{\em Homological local linking},
	Abstr. Appl. Anal. {\bf 3} (1998), 181--189.
	
	\bibitem{MR2640827}
	K.\ Perera, R.~P. Agarwal, D.\ O'Regan,
	{\em Morse theoretic aspects of {$p$}-{L}aplacian type operators},
	volume 161 of {\em Mathematical Surveys and Monographs}.
	American Mathematical Society, Providence, RI, 2010.
	
	\bibitem{zhikov86}
	 V.V.~Zhikov,
	{\em Averaging of functionals of the calculus of variations and elasticity theory},
	Izv. Akad. Nauk SSSR Ser. Mat. \textbf{50} (1986), 675--710.
	
	\bibitem{zhikov95} V.V.~Zhikov, 
	{\em  On Lavrentiev's Phenomenon}, 
	Russian J.\ Math.\ Phys.\ \textbf{3} (1995), 249--269.
	
	\bibitem{zhikov97} V.V.~Zhikov, 
	{\em On some variational problems}, 
	Russian J.\ Math.\ Phys.\ \textbf{5} (1997), 105--116.

\bibitem{zhikovBook} 
V.V.~Zhikov, S.M. Kozlov, O.A. Oleinik,
{\em Homogenization of differential operators
	and integral functionals}, 
Springer-Verlag, Berlin, 1994.

\end{thebibliography}
\end{document}